\documentclass[12pt]{amsart}
\usepackage[english]{babel}
\usepackage{tikz-cd}
\usepackage{cite}
\usepackage{url}
\usepackage{hyperref}
\usepackage{anysize}
\usepackage{graphicx}
\usepackage{fullpage}
\usepackage{colonequals}
\usepackage{amssymb}
\usepackage[leqno]{amsmath}
\usepackage{amscd}
\usepackage{amsthm}
\usepackage{amsmath}
\usepackage{enumitem}
\usepackage{xcolor}
\usepackage{marvosym}
\usepackage{soul} %para tachar
\usepackage{amsaddr}
\usepackage{amsthm} % lo agregué para cmabiar enumeración de teoremas
\usepackage{mathrsfs} %mas letras caligráficas

\marginsize{3cm}{3cm}{1cm}{2cm}
\usepackage{amsmath}
\usepackage[all,cmtip]{xy}
\newtheorem{theorem}{Theorem}[section]
\newtheorem{lemma}[theorem]{Lemma}
\newtheorem{Proposition}[theorem]{Proposition}
\newtheorem{Note}[theorem]{Note}
\newtheorem{Corollary}[theorem]{Corollary}
\newtheorem{Remark}[theorem]{Remark}
\newtheorem{Definition}[theorem]{Definition}

\numberwithin{equation}{section}

\title{On Larsen's conjecture on the ranks of Elliptic Curves}
\author{A. Hadavand}
\address{Department of Mathematics, Arak Branch, IAU, Arak, Iran.}
\email{hadavand@iau.ir; hadavand.math@gmail.com}
\date{\today}

\begin{document}
\begin{abstract}
Let $E$ be an elliptic curve over $\mathbb{Q}$ and $G=\langle\sigma_1, \dots, \sigma_n\rangle$ be a finitely generated subgroup of 
$\operatorname{Gal}(\overline{\mathbb{Q}}/ \mathbb{Q})$.  Larsen's conjecture  claims that the rank of the Mordell-Weil group $E(\overline{\mathbb{Q}}^G)$
is infinite where ${\overline{\mathbb Q}}^G$ is the $G$-fixed sub-field of $\overline{\mathbb Q}$. 
In this paper we prove the conjecture for the case in which  $\sigma_i$ for each $i=1, \dots, n$ is an element of some infinite families of elements of 
$\operatorname{Gal}(\overline{\mathbb{Q}}/ \mathbb{Q})$. 
\end{abstract}

%\begin{keyword}
%\texttt{ Modular curve \sep Elliptic curve\sep Heegner point \sep Absolute Galois group.}
%\MSC[2010] 11G05\sep  11R37.
%\end{keyword}
\maketitle

\section{Introduction}

Let $E$ be an elliptic curve over $\mathbb{Q}$ and $G=\langle\sigma_1, \dots, \sigma_n\rangle$ be a finitely generated subgroup of 
$\operatorname{Gal}(\overline{\mathbb{Q}}/ \mathbb{Q}).$   Michael Larsen  \cite{Lar} conjectured that 
the rank of the Mordell-Weil group  $E(\overline{\mathbb{Q}}^G)$ is infinite  where ${\overline{\mathbb Q}}^G$ is the $G$-fixed sub-field of $\overline{\mathbb Q}$.
 The conjecture has been proved by Bo-Hae Im \cite{Im} for the case $n=1$. Moreover in this case Bo-Hae Im and Larsen in \cite{Im-La1} proved that the rank of $E(\overline{\mathbb{Q}}^{\sigma_1})\otimes\mathbb{Q}$ is infinite when $E$ is an abelian variety over $\mathbb{Q}$. For the case $n=2$ some results are given in \cite{Alv, Im-Al}. See also \cite{Tim} for some evidence when $G=\langle\sigma_1, \dots, \sigma_n\rangle$.
 \\ Let $E$ be an elliptic curve over a global field $k$ and $\overline{k}$ be a separable closure of $k$. Breuer and Im \cite{Br-Im} proved that the rank of $E(\overline{k}^{\sigma})$ is infinite for every $\sigma\in\operatorname{Gal}(\overline{k}/k)$, in the following two cases. Firstly when $k$ is a global function field of odd characteristic and $E$ is parametrized by a Drinfeld modular curve, and secondly when $k$ is a totally real number field and $E$ is parametrized by Shimura curve. To find more and a survey about the conjecture see \cite{Im-La2}.\\ 
In this paper we prove the conjecture for the case in which  $\sigma_i$ for each $i=1, \dots, n$ is an element of an infinite family of elements of 
$\operatorname{Gal}(\overline{\mathbb{Q}}/ \mathbb{Q})$. 
 The strategy of the proof in \cite{Im} is using modularity and finding $\sigma_1$-fixed Heegner points which generate an infinite rank
subgroup of the Mordell-Weil group  $E(\overline{\mathbb{Q}}^{\sigma_1}).$

In this paper we find Heegner points over ``broad'' rather than ``deep''  Galois extensions. This gives us a sequence of independent Heegner points
over Hilbert class fields attached to an infinite family of imaginary quadratic fields, instead of a tower of ring class fields over an imaginary quadratic field. 
Indeed we present a family of Hilbert class fields $\{H_p\}_{p\in{{\mathcal A}_{N}}}$ where  $N$ is the conductor of $E$ and ${\mathcal A}_{N}$ is an infinite  family of prime numbers, and show that the rank of 
 $E((\prod_{p\in{{\mathcal A}^{'}}}{H_p})^G)$ is infinite for any  arbitrary infinite subset  ${\mathcal A}^{'}$  of ${\mathcal A}_{N}.$

\section{Heegner Points}

In this section $\mathcal{O}$ is an order in an imaginary quadratic field
 $k=\mathbb{Q}(\sqrt{-d})$ where $d$ is a positive integer number, and $k_{\mathcal{O}}$
is the ring class field of $k$ attached to $\mathcal{O}$. Moreover $N$ is a fixed 
positive integer number and $X_{0}(N)$ is the modular
curve. Recall that $X_0(N)$ is an algebraic curve over $\mathbb{C}$ which is the solution to the moduli space 
problem of classifying pairs $(E, C)$ where $E$ is an elliptic curve and $C$ is a cyclic subgroup 
of $E$ of order $N$. $X_0(N)$ possess a model over $\mathbb{Q}$, in this model, non-cuspidal points 
on the modular curve are given by a pair $(j(E), j(E'))$ in which $j(E)$ and $j(E')$ are the $j$-invariants
of elliptic curves $E$ and $E'$, respectively and there is an isogeny $\phi: E\rightarrow E'$ with 
$Ker(\phi)\cong \frac{\mathbb{Z}}{N\mathbb{Z}}$ (see page 18 \cite{Dar}).
\begin{Definition}
A Heegner point on $X_{0}(N)$ attached to $\mathcal{O}$ is a point $(j(E), j(E'))$ 
 such that $End(E)\cong End(E')\cong \mathcal{O}$, i.e., $E$ and $E'$ both have complex multiplication by $\mathcal{O}$.
\end{Definition}

We denote the set of Heegner points on $X_{0}(N)$ attached to $\mathcal{O}$ by $HP^{(N)}(\mathcal{O})$. 
Moreover we say that a Heegner point is attached to $k$ when $\mathcal{O}$ is the maximal order in $k$, and
write $HP^{(N)}(k)$ instead of $HP^{(N)}(\mathcal{O})$.
\begin{Proposition}\label{Pro1}
Let $gcd(Disc(\mathcal{O}), N)=1$ where $Disc(\mathcal{O})$ is the discriminant of $\mathcal{O}$. Then
$HP^{(N)}(\mathcal{O})$ is nonempty if and only if every prime dividing $N$ is split in $k$.
\end{Proposition}
\begin{proof}
See Proposition 3.8 in \cite{Dar}.
\end{proof}
By Proposition \ref{Pro1} we have the following definition.
\begin{Definition} 
 We say that $\mathcal{O}$ satisfies the Heegner hypothesis for $N$ if the following two conditions are satisfied:
\begin{itemize}
\item $\gcd(Disc(\mathcal{O}), N)=1.$
\item Every prime dividing $N$ is split in $k$.
\end{itemize}
\end{Definition}
\begin{theorem}\label{T3}
Let $y=(j(E), j(E'))\in HP^{(N)}(\mathcal{O})$, then 
\begin{enumerate}
\item $k_{\mathcal{O}}=k(j(E))$.
\item $HP^{(N)}(\mathcal{O})\subset X_0(N)(k_{\mathcal{O}})$, i.e., the Heegner points on $X_0(N)$ attached
to $\mathcal{O}$ are defined over $k_{\mathcal{O}}.$
\item Let $C_y=\{y^{\sigma}| \sigma\in \operatorname{Gal}(k_{\mathcal{O}}/k)\}$ be the full set of
Galois conjugates of $y$, then $C_y=HP^{(N)}(\mathcal{O})$.
\end{enumerate}
\end{theorem}
\begin{proof}
\begin{enumerate}
\item Since $E$ and $E'$ both have complex multiplication by $\mathcal{O}$, by Theorem 11.1 \cite{Cox}
we have $k_{\mathcal{O}}=k(j(E))=k(j(E'))$.
\item By 1 we have $y\in k_{\mathcal{O}}^2$.
\item See Corollary 5 in \cite{Sil}.
\end{enumerate}
\end{proof}

Let $E$ be an elliptic curve over $\mathbb{Q}$  of conductor $N$ and $\Phi_E:  X_0 (N)  \xrightarrow{} E$ be 
a modular parametrization of $E$ which is a map of algebraic curves defined over
   $\mathbb{Q}$  (see Theorem 8.8.4 in \cite{Dia}), we call $\Phi_E(HP^{(N)}(\mathcal{O}))$ the set of Heegner points
   on $E$ attached to $\mathcal{O}$.
\begin{theorem}\label{Theo2}
Let $E$ be an elliptic curve over $\mathbb{Q}$  of conductor $N$ and $\Phi_E:  X_0 (N)  \xrightarrow{} E$ be 
a modular parametrization of $E$. Let $y\in HP^{(N)}({\mathcal{O}})$, then $\Phi_E(y)\in E(k_{\mathcal{O}})$.
\end{theorem}
\begin{proof}
See Theorem 3.6 in \cite{Dar}.
\end{proof}
In the end of this section we stay some results on the structure of 
  $\operatorname{Gal}(k_{\mathcal{O}}/\mathbb{Q})$ which
we need. An element $\sigma\in \operatorname{Gal}(k_{\mathcal{O}}/\mathbb{Q})$ is called 
an involution if its restriction to $k$ is not the identity. For example by Lemma 5.28 \cite{Cox}  the complex
conjugation, $\tau$, is in $\operatorname{Gal}(k_{\mathcal{O}}/\mathbb{Q})$ is an involution with
$\tau^2=1.$
\begin{Proposition}\label{Pro2}
Let $\sigma\in \operatorname{Gal}(k_{\mathcal{O}}/\mathbb{Q})$ be an involution, then 
$$\sigma\gamma=\gamma^{-1}\sigma \ \ \ \text{for all}\ \ \gamma\in \operatorname{Gal}(k_{\mathcal{O}}/k). $$
\end{Proposition}
\begin{proof}
See Proposition 7 in \cite{Sil}.
\end{proof}

\begin{Corollary}\label{Cor1}
Let $\sigma\in \operatorname{Gal}(k_{\mathcal{O}}/\mathbb{Q})$ be an involution, then 
\begin{enumerate}
\item $\sigma^2=1$.
\item There are exactly $[k_{\mathcal{O}}:k]$ involutions in  $\operatorname{Gal}(k_{\mathcal{O}}/\mathbb{Q})$. 
\end{enumerate}
\end{Corollary}
\begin{proof}
\begin{enumerate}
\item Let $\tau$ be the complex conjugation, by Proposition \ref{Pro2}, $\tau\gamma=\gamma^{-1}\tau$ for all
$\gamma\in\operatorname{Gal}(k_{\mathcal{O}}/k)$. On the other hand since $k=\mathbb{Q}(\sqrt{-d})$, $\sigma|_k\neq1$
and $$\tau\sigma(\sqrt{-d})=\tau(-\sqrt{-d})=\sqrt{-d} \ \ \text{or } \ \ (\tau\sigma)|_k=1,$$
we have $\tau\sigma\in\operatorname{Gal}(k_{\mathcal{O}}/k)$. Therefore 
$$\sigma^2=\sigma1\sigma=\sigma\tau^2\sigma=\sigma\tau(\tau\sigma)=\sigma(\tau\sigma)^{-1}\tau=\sigma\sigma^{-1}\tau^{-1}\tau=1.$$

\item For any $\gamma\in \operatorname{Gal}(k_{\mathcal{O}}/k)$, $(\tau\gamma)|_k\neq1$ which means
 $\tau\gamma$ is an involution. On the other hand if $\phi$ is an involution, then

$$\phi=\tau^2\phi=\tau(\tau\phi)=\tau\gamma,$$
where $\gamma=\tau\phi\in \operatorname{Gal}(k_{\mathcal{O}}/k)$.
Thus there are exactly $\#\operatorname{Gal}(k_{\mathcal{O}}/k)=\#[k_{\mathcal{O}}: k]$ involutions in $ \operatorname{Gal}(k_{\mathcal{O}}/\mathbb{Q})$.
\end{enumerate}
\end{proof}

% % % % % % % % % % % % % % % % % % % % % % % % % % % % % % % % % % % % % % % % % % % % % % % % % % % % %

\section{Mordell-Weil groups with unbounded ranks}

Let $E$ be an elliptic curve over $\mathbb{Q}$  of conductor $N.$ 
We will present an infinite family ${{\mathcal A}_{N}}$ of prime numbers such that for any $p\in{{{\mathcal A}_{N}}}$ the imaginary quadratic field $ k_p :=\mathbb Q (\sqrt{-p})$
has odd class number and satisfies the Heegner hypothesis for $N$.
\begin{Proposition}\label{Pro3}
 If  $p\equiv 3 \ (\bmod\ 4)$ is a prime number then the class number of $\mathbb{Q}(\sqrt{-p})$ is odd.
\end{Proposition}
\begin{proof}
See Proposition 3.11 in \cite{Cox}.
\end{proof}
\begin{lemma}{\label{L1}}
Let  $N$ be a positive integer number.
There are infinitely many imaginary quadratic fields with odd class numbers that satisfy the Heegner hypothesis for $N.$
\end{lemma}

\begin{proof}
Let $N=2^{\alpha}p_1^{\alpha_1}p_2^{\alpha_2}\dots p_r^{\alpha_r}$ where $p_1, p_2, \dots, p_r$ are all distinct odd prime factors of $N$. Let $a$ be an integer number such that 
$a\equiv-1\ (\bmod\ 8)$  and $a\equiv-1\ (\bmod\  p_i)$ for each $i=1, 2, \dots, r$. Let $t=8p_1p_2\dots p_r+1$. Since $\gcd(8p_1p_2\dots p_r, at^2)=1$ by
 Dirichlet's theorem there are infinitely many primes of the form $p=(8p_1p_2\dots p_r)m_p + at^2$ where $m_p$ is an integer number for each prime $p$. 
For such a prime $p$ and  the imaginary quadratic field   $k_p=\mathbb{Q}(\sqrt{-p})$  we have $Disc(k_p)=-p$ and 
 $\gcd(Disc(k_p), N)=1$. Since $2\nmid Disc(k_p)$ and $Disc(k_p)\equiv1\ (\bmod\ 8)$, 2 is split in $k_p$ and moreover we have $(\frac{Disc(k_p)}{p_i})=1$ which implies that $p_i$ is split in $k_p$ for each $i=1, 2, \dots, r$. Therefore $k_p$ satisfies the Heegner hypothesis for $N$. On the other hand, since  $p\equiv 3 \ (\bmod\ 4)$, the class number of $k_p$  
 is an odd number by Proposition \ref{Pro3}.
\end{proof}
We denote the set of the primes presented in the proof of Lemma \ref{L1} by ${\mathcal A}_{N}$.
 For any $p\in {{\mathcal A}_{N}}$  let $k_p :=\mathbb Q (\sqrt{-p})$  and 
 $H_p$ be the Hilbert class field attached to $k_p$ and  $H_{E}:=\prod_{p\in {{\mathcal A}_{N}} }H_p$.\\

\begin{Proposition}\label{pro3}
Let $p\in{{{\mathcal A}_{N}}}$ and $\sigma\in\operatorname{Gal}(H_p/\mathbb{Q})$ be an involution.
Then there is $y_p\in HP^{(N)}(k_p)$ such that $\sigma(y_p)=y_p$.
\end{Proposition}
\begin{proof}
By Theorem \ref{T3} (3) let $HP^{(N)}({k_p})=\{ y_{1p}, y_{2p}, \dots, y_{h_pp} \}$ where $h_p$ is the class number of $k_p$.
Since $\sigma^2=1$, for each  $i=1,2, \dots, h_p$ we have $\sigma(y_{ip})=y_{j_{i}p}$ and
$\sigma (y_{j_{i}p})=y_{ip}$ or

 $$\sigma=\begin{pmatrix}\sqrt{-p}& -\sqrt{-p} & y_{1p} & y_{j_1p}& y_{2p} & y_{j_2p} &\dots & y_{h_pp} & y_{j_{h_p}p}\\
 -\sqrt{-p}& \sqrt{-p} & y_{j_{1}p} & y_{1p} & y_{j_{2}p} & y_{2p}&\dots & y_{j_{h_p}p}& y_{h_pp} \end{pmatrix}.$$
 Now since $h_p$ is  an odd number there is an $1\leq i\leq h_p$ such that $y_{ip}=y_{j_ip}$ or  $\sigma(y_{ip})=y_{ip}$.
\end{proof}
For all $p\in{{{\mathcal A}_{N}}}$  by Corollary \ref{Cor1} (2) there are  
$\psi_{1p}, \psi_{2p}, \dots, \psi_{h_pp}$ distinct involutions in  $\operatorname{Gal}(H_p/\mathbb{Q})$ 
where $h_p$ is  the class number of $k_p$. Let
$$\Sigma_j:=\{\sigma\in \operatorname{Gal}(\overline{\mathbb{Q}}/ \mathbb{Q})\ |\ \  \sigma|_{H_p}
=\psi_{jp}\ \ \text{for any} \ \  p\in{{{\mathcal A}_{N}}}\},$$
for a $1\leq j\leq h$ where $h=min\{h_p|\ \ p\in{{{\mathcal A}_{N}}}\}$.
$\Sigma_j$ is an infinite family of elements of $\operatorname{Gal}(\overline{\mathbb{Q}}/ \mathbb{Q})$ for any  $1\leq j\leq h$.\\

Let  $G=\langle\sigma_1, \dots, \sigma_n\rangle$ be a finitely generated subgroup of 
$\operatorname{Gal}(\overline{\mathbb{Q}}/ \mathbb{Q})$  where $\sigma_i$ is an element of $\Sigma_j$ for each $i=1, \dots, n$ and a  $1\leq j\leq h$,
it will be shown that the rank of $E({H_{E}}^G)$ is infinite where ${H_{E}}^G$  is the $G$-fixed sub-field of $H_E$.
To prove this we consider Heegner points attached to the Hilbert class fields $H_{p}$ for any $p$, and use 
the following theorem to obtain an infinite sequence of independent
non-torsion points in $E(H_{E})$ and then by using the modularity of $E$ it will be shown
that the rank of $E({H_{E}}^G)$ is infinite.\\
\begin{theorem}
Let $E$ be an elliptic curve over $\mathbb{Q}$ with conductor $N$, and let  $\Phi_E:  X_0 (N)  \xrightarrow{} E$  be a modular parametrization. Assume that we are given the following:
\begin{itemize}{\label{T1}}
\item $k_1, k_2,\dots , k_r$ are distinct  imaginary quadratic fields satisfying the Heegner hypothesis for N,
\item $h_1, h_2,\dots , h_r$ are the class numbers of $k_1, k_2,\dots k_r$ respectively,
\item $y_1, y_2, \dots , y_r$ are Heegner points on the modular curve attached to $k_1, k_2,\dots , k_r$
respectively,
\item $P_1, P_2,\dots , P_r$ are the attached Heegner points on $E$, $P_i:=\Phi_E (y_i)$.
\end{itemize}
There exists a constant $C = C(E,\Phi_E )$ such that, if the odd parts of the class numbers of 
$k_1, k_2,\dots , k_r$
are larger than $C$, then the points $P_1, P_2,\dots , P_r$ are independent in 
$E(\overline{\mathbb{Q}})/E_{tors}$.
\end{theorem}
\begin{proof}
See Theorem 1.1 in \cite{Sah}.
\end{proof}
   
Let $p_1,\dots , p_r\in{{\mathcal A}_{N}}$ and $k_{p_1}, \dots ,  k_{p_r}$  be the attached  
imaginary quadratic fields, respectively and $y_{p_i}$ be a Heegner point attached to $k_{p_i}$
for each $i=1,2, \dots  ,r$. Let $P_{p_i}:=\Phi_E (y_{p_i})$ be the attached Heegner point 
on $E$ (which is in $E(H_{p_{i}})$ by Theorem \ref{Theo2}) for each $i=1,2, \dots ,r.$ 
Since we have the following result due to Siegel (see page 149 in \cite{Cox}) 
$$\lim_{p\rightarrow +\infty}\frac{\log h_p}{\log p}=\frac{1}{2},$$
by removing finitely many $p\in{\mathcal A}_N$ if it is necessary we can assume that 
 $h_p>C = C(E,\Phi_E )$ for each $p\in{\mathcal A}_N$. On the other hand $h_{p_i}$ 
 the class number of $k_{p_i}$ is odd, therefore by Theorem \ref{T1} the points $P_{p_1},\dots , P_{p_r}$
 are independent in $E(\prod_{i=1}^{r}H_{p_i})/E_{tors} (\prod_{i=1}^{r}H_{p_i})$.
 Thus
the Heegner points in the set $\{P_p\}_{p\in{\mathcal{A}_N}}$ are non-torsion independent points 
in $E(H_{E})$. 

%%%%%%%%%%%%%%%%%%%%%%%%%%%%%%%%%%%%%%%%%%%%%%%%%%%%%%%%%%%%%%%%%%%%%%%%%
\begin{theorem}{\label{T4}}
Let $E$ be an elliptic curve over $\mathbb{Q}$ with conductor $N$ and $G=\langle\sigma_1, \dots, \sigma_n\rangle$
be a finitely generated subgroup of $\operatorname{Gal}(\overline{\mathbb{Q}}/ \mathbb{Q})$  
where $\sigma_i$  for each $i=1,2, \dots, n$  is an element of $\Sigma_j$ for a $1\leq j\leq h$. 
Then the rank of $E({H_E}^G)$  is infinite.
\end{theorem}
\begin{proof}
For any $i=1,2, \dots, n$ we have $\sigma_i\in\Sigma_j$ so $\sigma_i|_{H_p}=\psi_{jp}$
where $\psi_{jp}$ is an involution in $\operatorname{Gal}(H_P/ \mathbb{Q})$ for any $p\in{\mathcal A}_{N}$.
Let $p\in{\mathcal A}_{N}$, by Proposition \ref{pro3} there is $y_{jp}\in HP^{(N)}(k_p)$ 
such that $\psi_{jp}(y_{jp})=y_{jp}$. Thus for each  $i=1,2, \dots, n$
$$\sigma_i(y_{jp})=\psi_{jp}(y_{jp})=y_{jp},$$
or   $y_{jp}\in{{H_p}^G}$ which implies that 
 $$B_j:=\{y_{jp}:\ \  \text{for any}\ \ p\in{\mathcal A}_{N} \}\subset H_E^G.$$
 Now let $\Phi_E: X_0(N)\rightarrow E$ be a modular parametrization of $E$, by Theorem \ref{Theo2}
 $\Phi_E(B_j)\subset E(H_E)$ and since $\Phi_E$ is a map of algebraic curves defined over $\mathbb{Q}$
and $B_j\subset H_E^G$ indeed we have $\Phi_E(B_j)\subset E(H_E^G)$. Therefore
$\Phi_E(B_j)\subset \{P_p\}_{p\in{\mathcal{A}_N}}$ is an infinite set of non-torsion independent points in
$E(H_E^G)$ which implies that the rank of $E(H_E^G)$ is infinite.
\end{proof}
\begin{Note}
If one can show that $\#\Phi_E(HP^{(N)}(k_p))=\#HP^{(N)}(k_p)$, then one may  use the same idea
directly on the set points $P_p\in E(H_p)$ instead of doing it on the points $y_p$ on the modular curve.
\end{Note}

\begin{Corollary}
Let $E$ be an elliptic curve over $\mathbb{Q}$ and $G=\langle\sigma_1, \dots, \sigma_n\rangle$
 be a finitely generated subgroup of $\operatorname{Gal}(\overline{\mathbb{Q}}/ \mathbb{Q})$ 
 where $\sigma_i$ for each $i=1, \dots, n$  is an element of $\Sigma_j$ for a $1\leq j\leq h$. 
 Since ${H_E}^G < {\overline{\mathbb{Q}}}^G$   the rank of $E( {\overline{\mathbb{Q}}}^G)$  is infinite.
 Therefore  Larsen's conjecture holds for the case in which the generators of $G$ are in $\Sigma_j$ for a $1\leq j\leq h$.
\end{Corollary}
\begin{Remark}
By the presented method the rank of $E((\prod_{p\in{{\mathcal A}^{'}}}{H_p})^G)$ 
is infinite for any  arbitrary infinite subset  ${\mathcal A}^{'}$  of ${{\mathcal A}_{N}}.$\\
Let ${\mathcal A}_{N}={\{p_i\}}_{i\in{\mathbb N}}$ and ${\mathcal A}_i:={\mathcal A}_{N}-\{p_j| \ \  j=1, 2,\dots, i\}$ and $H_i:=\prod_{p\in{{\mathcal A}_i}}H_p$ 
then we have the following chain of the Mordell-Weil groups with unbounded ranks.\\
$$ \dots  <E({H_2}^G)<E({H_1}^G)<E({H_E}^G)<E( {\overline{\mathbb{Q}}}^G),$$
where the generators of $G$ are in   $\Sigma_j$ for a $1\leq j\leq h$.
\end{Remark}
\begin{theorem}
Let $E$ be an elliptic curve over $\mathbb{Q}$ with conductor $N$. Let 
$\{k_i\}_{i\in I}$ be an infinite family of imaginary quadratic fields  that satisfy 
the Heegner hypothesis for $N$. Then,
 \begin{itemize}
 \item if the odd part of the class number $h_i$ is larger than $C=C(E, \Phi_E)$ for each $i\in I$,
  then there are some infinite families $\Sigma$ of elements of
   $\operatorname{Gal}(\overline{\mathbb{Q}}/\mathbb{Q})$ such that for a finitely generated group 
  $G=\langle\sigma_1, \dots, \sigma_n\rangle$ in which $\sigma_j\in \Sigma$ for each $j=1, 2,\dots, n$  
  we have $Rank(E({H_E}^G))=\infty$ where $H_E=\prod_{i\in I} H_i$ and $H_i$ is the Hilbert class field attached to 
  $k_i$ for each $i\in I$.
  \item There is a chain 
  $$ \dots  <E({H_2}^G)<E({H_1}^G)<E({H_E}^G)<E( {\overline{\mathbb{Q}}}^G),$$
  where $Rank(E({H_j}^G))=\infty$ for each $j=1, 2, 3,\dots$.
  \end{itemize}
\end{theorem}
\begin{proof}
One can do the same as in the Proof of  Theorem \ref{T4}.
\end{proof}

\section*{Acknowledgments}
I would like to thank my family for supporting me in Baren-Dorud-Lorestan where I got the idea of this paper.

\selectlanguage{english}
\bibliographystyle{alpha}

\bibliography{mybibfile}

\end{document}